\documentclass[10pt,a4paper,twoside]{article}
\usepackage{amsthm,amsfonts,amssymb,amsmath,amscd,bezier,color}
\usepackage{graphicx}
\usepackage{hyperref}
\usepackage[all]{xy}

\newcommand{\Z}{\ensuremath{\mathbb{Z}}}

\newcommand{\C}{\ensuremath{\mathfrak{c}}}

\newcommand{\RP}{\ensuremath{\mathbb{R}\mathrm{P}}}
\newcommand{\Pe}{\ensuremath{\mathcal{P}}}

\newcommand{\q}{\ensuremath{\mathfrak{q}}}
\newcommand{\p}{\ensuremath{\mathfrak{p}}}

\newtheorem{definition}{Definition}[section]
\newtheorem{proposition}[definition]{Proposition}
\newtheorem{lemma}[definition]{Lemma}

\newtheorem{theorem}[definition]{Theorem}
\newtheorem{corollary}[definition]{Corollary}
\newtheorem{example}[definition]{Example}

\begin{document}

\thispagestyle{empty}

\begin{center}

{\Large Strong surjections from two-complexes with odd order \\ \vspace{2mm} top-cohomology onto the projective plane}

\vspace{10mm}

{\large Marcio C.\,Fenille, Daciberg L.\,Gon\c calves and  Oziride M.\,Neto}

\vspace{7mm}

\end{center}

\noindent{\bf Abstract:} Given a finite and connected two-dimensional $CW$-complex $K$ with fundamental group $\Pi$ and second integer cohomology group $H^2(K;\Z)$ finite of odd order, we prove that: (1) for each local integer coefficient system $\alpha:\Pi\to{\rm Aut}(\Z)$ over $K$, the corresponding twisted cohomology group $H^2(K;_{\alpha}\!\Z)$ is finite of odd order, we say order $\C^{\ast}(\alpha)$, and there exists a natural function -- which resemble that one defined by the twisted degree -- from the set $[K;\RP^2]_{\alpha}^{\ast}$ of the based homotopy classes of based maps inducing $\alpha$ on $\pi_1$ into $H^2(K;_{\alpha}\!\Z)$, which is a bijection; (2) the set $[K;\RP^2]_{\alpha}$ of the (free) homotopy classes of based maps inducing $\alpha$ on $\pi_1$ is finite of order $\C(\alpha)=(\C^{\ast}(\alpha)+1)/2$; (3) all but one of the homotopy classes $[f]\in[K;\RP^2]_{\alpha}$ are strongly surjective, and they are characterized by the non-nullity of the induced homomorphism $f^{\ast}:H^2(\RP^2;_{\varrho}\!\Z)\to H^2(K;_{\alpha}\!\Z)$, where $\varrho$ is the nontrivial local integer coefficient system over the projective plane. Also some calculations of the groups $H^2(K;_{\alpha}\!\Z)$ are provided for several two-complexes $K$ and actions $\alpha$, allowing to compare $H^2(K;\Z)$ and $H^2(K;_{\alpha}\!\Z)$ for nontrivial $\alpha$.

\vspace{5mm}


\vspace{3mm}


   
\section{Introduction and main theorem}\label{Section-Introduction}

\hspace{4mm} We study a problem related to both, homotopy theory and topological root theory.

From the viewpoint of homotopy theory, the problem can be considered as a type of Hopf-Whitney Classification Theorem \cite[Corollary 6.19, p.\,244]{Whitehead}, which  states that for a given finite and connected two-dimensional $CW$ complex $K$ (a two-complex, for short), there exists a bijection between the set $[K;S^2]$ of the free homotopy classes of maps from $K$ into the two-sphere $S^2$ and the integer cohomology group $H^2(K;\Z)$.

One could search for a result which describe the set 
$[K;\RP^2]$ in terms of the cohomology of $K$ with some local coefficient system. Certainly $[K;\RP^2]$ can not be identified  with $H^2(K;\Z)$ because the sets $[K;S^2]$ and 
$[K;\RP^2]$ do not have the same cardinality in general.  

Suppose that the two-complex $K$ is such that $[K;\RP^2]=0$. Then all maps $f:K\to\RP^2$ are homotopic to a constant map, and so all of them lifts, through the double covering map $\p:S^2\to\RP^2$, to maps $\tilde{f}:K\to S^2$ which are also homotopic to a constant map. Hence $[K;S^2]=0$, which forces $H^2(K;\Z)=0$. Therefore, $[K;\RP^2]=0$ implies $H^2(K;\Z)=0$. The opposite implication is not true; for instance, $K=S^1$. Theorem 1.1 of \cite{Fenille-Trivial-Class} gives four equivalent assumptions on the two-complex $K$ in order to the implication $H^2(K;\Z)=0\Rightarrow[K;\RP^2]=0$ to be true. Therefore, we have a (first) natural question: given a two-complex $K$ with $H^2(K;\Z)=0$, we ask how to describe the set $[K;\RP^2]$  when $K$ does not necessarily satisfy the assumptions of the theorem above referenced. In fact, in this work we consider the question in the more general setting where the group $H^2(K;\Z)$ is finite of odd order.

Another goal of this work is to study problems related to the topological root theory for maps $f:K \to \RP^2$. For this, we recall the following definition:

\begin{definition}{\rm 
For a {\it strongly surjective} map (or a {\it strong surjection}) we mean a map whose free homotopy class contains just surjective maps. Such a free homotopy class is also said to be {\it strongly surjective}.}
\end{definition}

By the Cellular Approximation Theorem, any map from a one-complex into a closed surface is not strongly surjective. Additionally, by the Universal Coefficient Theorem for Cohomology, a two-complex $K$ with $H^2(K;\Z)=0$ is (co)homological like a one-complex, since the homology groups $H_2(K;\Z)=0$ and $H_1(K;\Z)$ is torsion free. Therefore, it is natural to ask if there exists a two-complex $K$ with $H^2(K;\Z)=0$ and a strong surjection from $K$ into a closed surface. The first answer for this question was given in \cite{Fenille-Toro}, where the first author built a countable collection of two-complexes $K$ with $H^2(K;\Z)=0$ and, for each of them, a strong surjection onto the torus $S^1\times S^1$. By composing such map with the double covering map we get a strong surjection onto the Klein bottle.

Once an affirmative answer for the question was given, an ensuing problem is to know for what closed surfaces the question has an affirmative answer. The first attempt to answer the problem for the projective plane found a negative partial answer: Theorem 1.1 of \cite{Fenille-One-Relator-RP2} states that for a two-complex $K$ with just a two-cell and a map $f:K\to\RP^2$ inducing the homomorphism $\alpha$ on fundamental groups, if the twisted cohomology group $H^2(K;_{\alpha}\!\Z)=0$, then $f$ is not strongly surjective. But an affirmative answer was given in \cite{Fenille-Daciberg}, where the problem was approached for a particular countable collection of two-complexes, again with a single two-cell each, and it was described representatives for all free and based homotopy classes of maps from a two-complex of such collection into $\RP^2$. Moreover, it was provided a classification of the homotopy classes according to the property of to be strongly surjective. Therefore, article \cite{Fenille-Daciberg} moves on back to the viewpoint of homotopy theory.

In this paper we solve the following problem: given a two-complex $K$ whose second integer cohomology group  $H^2(K;\Z)$ is finite of odd order, describe  the set $[K;\RP^2]$ and provide a classification of the homotopy classes $[f]\in[K;\RP^2]$ according to the property of to be strongly surjective. We extends the results of \cite{Fenille-Daciberg} for all two-complexes $K$ for which $H^2(K;\Z)$ is finite of odd order. We count the free and based homotopy classes of maps from such a two-complex $K$ into $\RP^2$ and we prove that for each homomorphism $\alpha: \pi_1(K)\to\pi_1(\RP^2)$, there exists just one homotopy class of maps inducing $\alpha$ on fundamental groups which is not strongly surjective.

Throughout the text, we use the notations, simplifications and terminologies below.

We call a finite and connected two-dimensional $CW$ complex by a two-complex; we simplify $f$ is a continuous map by $f$ is a map; we consider the cyclic group $\Z_2=\{1,-1\}$ with its multiplicative structure; where appropriated, we identify an automorphism $\tau\in{\rm Aut}(\Z)$ with its value $\tau(1)$; the order of a group $G$ is indicated by $|G|$. 

Given a two-complex $K$, let $[K;\RP^2]^{\ast}_{\alpha}$ and $[K;\RP^2]_{\alpha}$ be the sets of the based and free homotopy classes $[f]^{\ast}$ and $[f]$, respectively, of based maps $f:K\to\RP^2$ inducing the homomorphism $\alpha$ on fundamental groups.

For each two-complex $K$ given {\it a priori}, its fundamental group will be denoted by $\Pi$. By the natural identification ${\rm Aut}(\Z)\approx\Z_2$, each homomorphism $\alpha\in\hom(\Pi;\Z_2)$ can be seen as a twisted integer coefficient system $\alpha:\Pi\to{\rm Aut}(\Z)$ over $K$ and we consider the corresponding twisted cohomology group $H^2(K;_{\alpha}\!\Z)$.

For the particular case $K=\RP^2$, we consider the only nontrivial homomorphism $\varrho:\pi_1(\RP^2)\to{\rm Aut}(\Z)$ and the corresponding twisted cohomology group $H^2(\RP^2;_{\varrho}\!\Z)$, which is well known to be the infinite cyclic group, so that we can choose appropriately a generator $\nu\in H^2(\RP^2;_{\varrho}\!\Z)$. Therefore, given  $[f]^{\ast}\in[K;\RP^2]_{\alpha}^{\ast}$ we can consider the induced homomorphism $f^{\ast}:H^2(\RP^2;_{\varrho}\!\Z)\to H^2(K;_{\alpha}\!\Z)$ which is defined by its value $f^{\ast}(\nu)$.

With these notations and terminologies we state our main theorem:

\begin{theorem}[Main Theorem]\label{Theorem-Main}
Let $K$ be a two-complex with fundamental group $\Pi$ and such that $H^2(K;\Z)$ is finite of odd order. Then, for each $\alpha\in\hom(\Pi;\Z_2)$, the twisted cohomology group $H^2(K;_{\alpha}\!\Z)$ is finite of odd order, we say order $\C^{\ast}(\alpha)$, and we have:
\begin{enumerate}
    \item[{\bf 1.}] The function $[f]^{\ast}\mapsto f^{\ast}(\nu)$ provides a bijection from $[K;\RP^2]_{\alpha}^{\ast}$ onto $H^2(K;_{\alpha}\!\Z)$, and so the set $[K;\RP^2]_{\alpha}^{\ast}$ is finite of cardinality $\C^{\ast}(\alpha)$.
    \item[{\bf 2.}] The set $[K;\RP^2]_{\alpha}$ is finite of cardinality $\C(\alpha)=(\C^{\ast}(\alpha)+1)/2$.
    \item[{\bf 3.}] All but one of the homotopy classes $[f]\in[K;\RP^2]_{\alpha}$ are strongly surjective, and they are characterized by the non-nullity of $f^{\ast}(\nu)$.
\end{enumerate}
\end{theorem}

Of course, Theorem \ref{Theorem-Main} includes the case in which $H^2(K;\Z)=0$ and $\alpha_0\in\hom(\Pi;\Z_2)$ is the trivial homomorphism; in this case, $\C^{\ast}(\alpha_0)=\C(\alpha_0)=1$ and the unique homotopy class $[f]\in[K;\RP^2]_{\alpha_0}$ is not strongly surjective (it is the class of the constant maps). 

In addition to extending the results of \cite{Fenille-Daciberg}, Theorem \ref{Theorem-Main} provides an extension and the converse of Theorem 1.1 of \cite{Fenille-One-Relator-RP2}.

Theorem \ref{Theorem-Main} is proved throughout the text. We highlight the steps given in each section. In Section \ref{Section-Twisted-Cohomology} we prove (Theorem \ref{Theorem-Odd-Order}) that for each $\alpha\in\hom(\Pi;\Z_2)$, the group $H^2(K;_{\alpha}\!\Z)$ is finite of odd order. In Section \ref{Section-f0} we prove (Theorem \ref{Theorem-Exists-f0}) that each $\alpha\in\hom(\Pi;\Z_2)$ can be realized as the homomorphism induced on fundamental groups by a non-surjective map $f_0:K\to\RP^2$. In Section \ref{Section-Bijections} we establish (Theorem \ref{Theorem-Bijection}), for each $\alpha\in\hom(\Pi;\Z_2)$, a certain function $\Phi:[K;\RP^2]_{\alpha}^{\ast}\to H^2(K;_{\alpha}\!\Z)$ is a bijection. In Section \ref{Section-Actions-Proof} we prove (Theorem \ref{Theorem-Action}) that a given based map $f:K\to\RP^2$ is strongly surjective if and only if its based homotopy class $[f]^{\ast}$ is not fixed by the natural action of $\pi_1(\RP^2)$, and we prove (Corollary \ref{Corollary-Action-Classes}) that such action corresponds, via the bijection $\Phi$, to the multiplication by $-1$ in $H^2(K;_{\alpha}\!\Z)$. In the end of Section \ref{Section-Actions-Proof} we present the proof of the Main Theorem itself.

In Section \ref{Section-Appendix} we give some results about the order of the twisted cohomology group $H^2(K;_{\alpha}\!\Z)$ when we vary both, the two-complexes $K$ with $H^2(K;\Z)=0$ and the possible homomorphisms $\alpha$ in $\hom(\Pi;\Z_2)$. The main result of the section is Proposition \ref{Proposition-All-Odd-Orders}. We highlight that such result provides examples of two-complexes $K$ with $H^2(K;\Z)=0$ for which there exists strong surjections from $K$ onto $\RP^2$.

We conclude this introduction by observing that the statement of item 1 of Theorem 1.2 (Main Theorem) does not hold if $H^2(K;_{\alpha}\!\Z)$ is not finite of odd order for some $\alpha$. The study of the questions considered in this work under the hypothesis that $H^2(K;_{\alpha}\!\Z)$ is not finite of odd order is a work in progress.


\section{The second twisted cohomology group}\label{Section-Twisted-Cohomology}

\hspace{4mm} Let $K$ be a two-complex with fundamental group $\Pi$. Each homomorphism $\alpha\in\hom(\Pi;\Z_2)$ can be seen as a local integer coefficient system $\alpha:\Pi\to{\rm Aut}(\Z)$ over $K$, and so it is defined the second twisted cohomology group $H^2(K;_{\alpha}\!\Z)$ with local coefficient system $\alpha$.

In \cite[Section 3]{Fenille-One-Relator-RP2} it is presented a way to compute this group, by using Fox-Reidemeister derivative. We present briefly the procedure here.

By \cite[Theorem 1.9]{Livro-Verde-Chapter-II}, the skeleton-pair $(K,K^1)$ is homotopy equivalent the that of the model two-complex $K_{\Pe}$ induced by a group presentation $\Pe=\langle x_1,\ldots,x_n\,|\,r_1,\ldots,r_m \rangle$ of the fundamental group $\Pi$. Let $\Omega:F(x_1,\ldots,x_n)\to\Pi$ be the corresponding epimorphism.

We define the following $m\times n$ integer matrices:
\begin{itemize}
    \item $\Delta=(\delta_{ij})$, in which $\delta_{ij}$ is the sum of the powers of the letter $x_j$ in the relator word $r_i$.
    \item $\Delta^{\!\alpha}=(\delta_{ij}^{\alpha})$, in which $\delta_{ij}^{\alpha}=\|\partial r_i/\partial x_j\|_{\alpha}$, where $\partial$ indicates the Fox-Reidemeister derivative and $\|\cdot\|_{\alpha}$ is the composition of the $\alpha$-augmentation function $\xi_{\alpha}:\Z\Pi\to\Z$ with the natural extension $\|\cdot\|:\Z F(x_1,\ldots,x_n)\to\Z\Pi$ of $\Omega$ on group rings.
\end{itemize}

Ir order to prove the main result of this section (Theorem \ref{Theorem-Odd-Order}), we use the following useful fact from \cite[Lemma 3.1]{Fenille-One-Relator-RP2}: for each $\alpha\in\hom(\Pi;\Z_2)$, one has \begin{equation}\label{Equation-Delta}
    \Delta^{\!\alpha}\equiv\Delta\,{\rm mod}\,2.
\end{equation}

In what follows, we consider the {\it cokernel} of a group homomorphism $h:G\to G'$ as being the quotient group ${\rm coker}(h)=G'/{\rm im}(h)$, as long  ${\rm im}(h)\subset G'$ is a normal subgroup.      

Let $\Lambda^{\!\alpha}:\Z^n\to\Z^m$ be the homomorphism whose matrix, relative to the canonical bases, is $\Delta^{\!\alpha}$. From \cite[Section 3]{Fenille-One-Relator-RP2}, we have
\begin{equation}\label{Equation-H2}
    H^2(K;_{\alpha}\!\Z)\approx{\rm coker}(\Lambda^{\!\alpha}).
\end{equation}

Of course, if $\alpha$ is the trivial homomorphism, then $\Delta^{\!\alpha}=\Delta$ and $H^2(K;_{\alpha}\!\Z)=H^2(K;\Z)$.

\begin{theorem}\label{Theorem-Odd-Order}
If $H^2(K;_{\beta}\!\Z)$ is finite of odd order for some $\beta\in\hom(\Pi;\Z)$, then so is $H^2(K;_{\alpha}\!\Z)$ for all $\alpha\in\hom(\Pi;\Z)$.
\end{theorem}
\begin{proof}
Given $\alpha,\beta\in\hom(\Pi;\Z_2)$, let consider the corresponding matrices $\Delta^{\!\alpha},\Delta^{\!\beta}$ and homomorphisms $\Lambda^{\!\alpha},\Lambda^{\!\beta}:\Z^n\to\Z^m$ defined as above. We just need to prove that if ${\rm coker}(\Lambda^{\!\beta})$ is finite of odd order, then so is ${\rm coker}(\Lambda^{\!\alpha})$. Since $\Delta^{\!\alpha}\equiv\Delta^{\beta}\,{\rm mod}\,2$ (as an immediate consequence of  Equation (\ref{Equation-Delta})), the result follows from Lemma \ref{Lemma-Coker} below for $p=2$.
\end{proof}

\begin{lemma}\label{Lemma-Coker}
Let $\mathbf{a},\mathbf{b}:\Z^n\to\Z^m$ be homomorphisms with matrix $A$ and $B$, respectively, relative to the same pair of bases, and let $p\geq2$ be an integer. If $A\equiv B\,{\rm mod}\,p$ and ${\rm coker}(\mathbf{a})$ is finite of order coprime with $p$, then so is ${\rm coker}(\mathbf{b})$.
\end{lemma}
\begin{proof}
Given a homomorphism $h:\Z^n\to\Z^m$, then ${\rm coker}(h)$ is finite of order coprime with $p$ if and only if ${\rm coker}(h)\otimes\Z_p=0$. Tensorizing the exact sequence $\Z^n\stackrel{h}{\rightarrow}\Z^m\stackrel{\pi}{\rightarrow}{\rm coker}(h)\to0$ with $\Z_p$ we obtain the exact sequence $$\Z^n\otimes\Z_p\stackrel{h\otimes1}{\longrightarrow}\Z^m\otimes\Z_p\stackrel{\pi\otimes1}{\longrightarrow}{\rm coker}(h)\otimes\Z_p\to0.$$ 

It follows that ${\rm coker}(h)$ is finite of order coprime with $p$ if and only if $h\otimes1$ is surjective.

Of course, fixed bases for $\Z^n$ and $\Z^m$ produce bases for $\Z^n\otimes\Z_p\approx\Z_p^n$ and $\Z^m\otimes\Z_p\approx\Z_p^m$ in such a way that if $C=(c_{ij})$ is the matrix of $h$ relative to the fixed bases for $\Z^n$ and $\Z^m$, then the matrix of $h\otimes 1$, relative to the corresponding bases for $\Z_p^n$ and $\Z_p^m$, is $\bar{C}=(\bar{c}_{ij})$, where the bar indicates the congruence class module $p$.

Now, consider the given homomorphisms  $\mathbf{a},\mathbf{b}:\Z^n\to\Z^m$. Since $A\equiv B\,{\rm mod}\,p$, we have $\bar{A}=\bar{B}$. It follows that $\mathbf{a}\otimes1=\mathbf{b}\otimes1$ and the result follows.
\end{proof}


\section{A special realization of homomorphisms}\label{Section-f0}

\hspace{4mm} As in the previous section, let $K$ be a two-complex whose fundamental group $\Pi$ is presented by $\Pe=\langle x_1,\ldots,x_n\,|\,r_1,\ldots,r_m \rangle$, so that $K$ is homotopy equivalent to the model two-complex $K_{\Pe}$. We have the following result:

\begin{lemma}\label{Lemma-hom-lift}
Suppose $H^2(K;\Z)$ is finite of odd order. Then $m\leq n$, $\hom(\Pi;\Z_2)\approx\Z_2^{n-m}$ and each $\alpha\in\hom(\Pi;\Z_2)$ lifts through the natural epimorphism $\q:\Z\to\Z_2$.
\end{lemma}
\begin{proof}
That $m\leq n$ follows from the characterization of $H^2(K;\Z)$ given in Equation (\ref{Equation-H2}).

By the Universal Coefficient Theorem for Cohomology \cite[Theorem 3.2, p.\,195]{Hatcher}, since $H^2(K;\Z)$ is finite, we have $H_2(K)=0$ and $H_1(K)\approx\Z^k\oplus\mathcal{T}$, for some integer $k$, in which the torsion group $\mathcal{T}\approx
H^2(K;\Z)$. This implies that $H_1(K;\Z_2)$ has rank $k$ over $\Z_2$.

The Euler characteristic of $K$ is $1+m-n$ (by the cellular decomposition of $K_{\Pe}$) and we have $H_0(K;\Z_2)\approx\Z_2$ and $H_2(K;\Z_2)=0$, since $H_2(K)=0$ and the torsion subgroup of $H_1(K)$ has odd order. It follows that $1+m-n=1-{\rm rank}_{\Z_2}\!(H_1(K;\Z_2))+0$, which implies that  $H_1(K;\Z_2)$ has rank $n-m$ over $\Z_2$. Therefore $k=n-m$.

Since $\Z_2$ is abelian, the composition with the abelianization homomorphism $\rho:\Pi\to\Pi^{ab}$ provides an isomorphism $\hom(\Pi^{ab};\Z_2)\approx\hom(\Pi;\Z_2)$. Since $\Pi^{ab}\approx H_1(K)$ and $\mathcal{T}$ is finite of odd order, which forces each homomorphism from $\Pi^{ab}$ into $\Z_2$ to map $\mathcal{T}$ to $1$, we have $$\hom(\Pi;\Z_2)\approx\hom(\Z^k\oplus\mathcal{T};\Z_2)\approx\hom(\Z^k;\Z_2)\approx\Z_2^k.$$

To complete the proof, given $\alpha\in\hom(\Pi;\Z_2)$, let take $\bar{\alpha}\in(\Pi^{ab};\Z_2)$ such that $\bar{\alpha}\circ\rho=\alpha$. Since $\bar{\alpha}$ maps
$\mathcal{T}$ to $1$ and $\q:\Z\to\Z_2$ is surjective, there exists a homomorphism $\hat{\alpha}:\Pi^{ab}\to\Z$ such that
$\q\circ\hat{\alpha}=\bar{\alpha}$. Therefore, $\widetilde{\alpha}=\hat{\alpha}\circ\rho:\Pi\to\Z$ is a lifting of $\alpha$ through $\q$.
\end{proof}

\begin{theorem}\label{Theorem-Exists-f0}
Suppose $H^2(K;\Z)$ is finite of odd order. Given $\alpha\in\hom(\Pi;\Z_2)$, there exists a non-surjective map $f_0:K\to\RP^2$ inducing $\alpha$ on fundamental groups.
\end{theorem}
\begin{proof}
By Lemma \ref{Lemma-hom-lift}, a given homomorphism $\alpha:\Pi\to\Z_2$ lifts to $\widetilde{\alpha}:\Pi\to\Z$ through $\q:\Z\to\Z_2$. By \cite[Lemma 2.5]{Fenille-Dide}, there exists a cellular map $f_0^1:K\to S^1$ such that $\widetilde{\alpha}=(f_0^1)_{\#}:\Pi\to\Z$. Consider the map $f_0:K\to\RP^2$ given by the composition $f_0=l\circ f_0^1$, in which $l:S^1\hookrightarrow\RP^2$ is the skeleton inclusion. Then $(f_0)_{\#}=\alpha$.
\end{proof}


\section{Bijections with based homotopy classes}\label{Section-Bijections}

\hspace{4mm} Let $K$ be a two-complex with fundamental group $\Pi$ and such that $H^2(K;\Z)$ is finite of odd order. Take, once and for all (for this section), a homomorphism $\alpha\in\hom(\Pi;\Z_2)$.

We consider the non-trivial homomorphism $\varrho:\pi_1(\RP^2)\to{\rm Aut}(\Z)$ and the  twisted 
cohomology group $H^2(\RP^2;_{\varrho}\!\Z)$,  which is well known to be the infinite cyclic group.

Each map $f:K\to\RP^2$ inducing $\alpha$ on the fundamental groups induces a homomorphism $$f^{\ast}:H^2(\RP^2;_{\varrho}\!\Z)\to H^2(K;_{\alpha}\!\Z).$$

Let us consider, in particular, a non-surjective map $f_0:K\to\RP^2$ inducing $\alpha$ on fundamental groups, so that $f_0^{\ast}=0$. Such a map exists, by Theorem \ref{Theorem-Exists-f0}.

Following \cite[Subsections 3.2 -- 4.2]{Livro-Verde-Chapter-II}, each twisted cohomology class $\{\gamma\}\in H^2(K;_{\alpha}\!\Z)$ is represented\footnote{In \cite[Theorem 4.12]{Livro-Verde-Chapter-II}, it is used the notation $H^2(\widetilde{K};_{\alpha}\!\Z)$ to indicate the equivariant cohomology of the universal covering space $\widetilde{K}$ of $K$, what in turn corresponds to $H^2(K;_{\alpha}\!\Z)$ in our notation.} by a $\Z\Pi$-module homomorphism (a twisted cochain) $$\gamma:C^2(\widetilde{K})\,\to\,  _{\alpha}\pi_2(\RP^2)\approx\,_{\alpha}\Z.$$

Given such a cochain $\gamma$, we follow \cite[Section 3.2]{Livro-Verde-Chapter-II} to defined the so-called {\it spherical modification} of $f_0$ by $\gamma$ as being the map $f_0^{\gamma}:K\to\RP^2$ which acts just like $f_0$, except that it preliminarily pinches a circle in each two-cell $c^2$ on $K$ to create a two-sphere $S^2$ based at the boundary of $c^2$, which it maps via a suitable based representative $S^2\to S^2\stackrel{\p\,\,}{\to}\RP^2$ of $\gamma(\tilde{c}^2)\in \,_{\alpha}\pi_2(\RP^2)\equiv [S^2;S^2]$, where $\tilde{c}^2$ is a preferred two-cell above $c^2$. 

The map $f_0^{\gamma}$ induces the homomorphism $\alpha$ on fundamental groups and its based homotopy class is determined by the twisted cochain $\gamma$. By \cite[Lemma 3.5]{Livro-Verde-Chapter-II}, the corresponding equivariant maps induced by $f_0$ and $f_0^{\gamma}$ coincide through dimension $1$ and their difference in dimension $2$ is the twisted cochain $\gamma$.

Since $f_0^{\ast}=0$ and, moreover, the covering map $\p:S^2\to\RP^2$ induces multiplication by $2$ on twisted cohomology groups, it follows that the induced homomorphism $$(f_0^{\gamma})^{\ast}:H^2(\RP^2;_{\varrho}\!\Z)\to H^2(K;_{\alpha}\!\Z) \quad\textrm{is defined by}\quad (f_0^{\gamma})^{\ast}(\nu)=2\{\gamma\},$$ where $\nu$ is a preferred generator of the infinite cyclic group $H^2(\RP^2;_{\varrho}\!\Z)$.

By \cite[Theorem 4.12]{Livro-Verde-Chapter-II}, we have a bijection $$f_0^{(\,\cdot\,)}:H^2(K;_{\alpha}\!\Z)\to[K;\RP^2]_{\alpha}^{\ast} \quad{\rm given \ by}\quad \{\gamma\}\mapsto[f_0^{\gamma}]^{\ast}.$$

Since the group $H^2(K;_{\alpha}\!\Z)$ is finite of odd order, the multiplication by $2$ on this group provides an isomorphism. Therefore, the function $$\Phi:[K;\RP^2]_{\alpha}^{\ast}\to H^2(K;_{\alpha}\!\Z) \quad\textrm{given by}\quad [f]^{\ast}\mapsto f^{\ast}(\nu)$$ is a bijection making commutative the diagram below, in which the ascending vertical arrow is the isomorphism given by the multiplication by $2$ and the descending one is its inverse:

\begin{center}
\begin{tabular}{c}\xymatrix{ & H^2(K;_{\alpha}\!\Z) \ar@/^0.15cm/[d]^-{\frac{1}{2}} \\ [K;\RP^2]^{\ast}_{\alpha} \ar[ru]^-{\Phi}  & H^2(K;_{\alpha}\!\Z) \ar@/^0.15cm/[u]^-{2} \ar[l]^-{f_0^{(\,\cdot\,)}} }
\end{tabular}
\end{center}

We have proved the following theorem:

\begin{theorem}\label{Theorem-Bijection}
Let $K$ be a two-complex with fundamental group $\Pi$ and such that $H^2(K;\Z)$ is finite of odd order. For each homomorphism $\alpha\in\hom(\Pi;\Z_2)$, the function $\Phi:[f]^{\ast}\mapsto f^{\ast}(\nu)$ is a bijection from $[K;\RP^2]_{\alpha}^{\ast}$ onto $H^2(K;_{\alpha}\!\Z)$.
\end{theorem}

No longer Theorem \ref{Theorem-Bijection} holds for a two-complex $K$ for which $H^2(K;\Z)$ is finite of even order. For instance, if $K=\RP^2$ and $\varrho\in\hom(\Z_2;\Z_2)$ is the identity homomorphism, then $\Phi$ is injective but not  surjective; in fact $H^2(\RP^2;_{\varrho}\!\Z)\approx\Z$ and the image of $\Phi$ is the set of the odd integers; see \cite[Section 3]{Fenille-Daciberg} or \cite[Proposition 2.1]{Daciberg-Spreafico}. The crucial fact is that there is not a map $f_0:\RP^2\to\RP^2$ such that $(f_0)_{\#}=\varrho$ and $f_0^{\ast}=0$.


\section{Actions on strong surjections}\label{Section-Actions-Proof}

\hspace{4mm} Given a two-complex $K$ with fundamental group $\Pi$, we recall that the fundamental group $\pi_1(\RP^2)$ acts on the set $[K;\RP^2]^{\ast}$ of based homotopy classes in such a way that the set $[K;\RP^2]$ of free homotopy classes corresponds to the quotient set of $[K;\RP^2]^{\ast}$ by this action. Moreover, by \cite[Lemma 2.1]{Fenille-Daciberg}, for each $\alpha\in\hom(\Pi;\Z_2)$, the subset $[K;\RP^2]_{\alpha}^{\ast}$ of $[K;\RP^2]^{\ast}$ is invariant by this action. 

The following theorem characterizes, via the effect of the action of $\pi_1(\RP^2)$, the strongly surjective homotopy classes.

\begin{theorem}\label{Theorem-Action}
Let $K$ be a two-complex with fundamental group $\Pi$ and such that $H^2(K;\Z)$ is finite of odd order. A given based map $f:K\to\RP^2$ is strongly surjective if and only if its based homotopy class $[f]^{\ast}$ is not fixed by the action pf $\pi_1(\RP^2)$ on the set $[K;\RP^2]^{\ast}$.
\end{theorem}
\begin{proof}
Let $f:K\to\RP^2$ be a based map and put $\alpha=f_{\#}:\Pi\to\Z_2$. By Theorem \ref{Theorem-Bijection}, the function $[f]^{\ast}\mapsto f^{\ast}(\nu)$ is a bijection from $[K;\RP^2]_{\alpha}^{\ast}$ onto $H^2(K;_{\alpha}\!\Z)$.

Following \cite[Section 3]{Fenille-Daciberg}, let $h_{-1}:\RP^2\to\RP^2$ be a self-map of \emph{twisted degree} $-1$. Define $f_{-1}:K\to\RP^2$ by the composite $f_{-1}=h_{-1}\circ f$. Since $h_{-1}$ induces the identity isomorphism on fundamental groups, we have $[f_{-1}]^{\ast}\in [K;\RP^2]_{\alpha}^{\ast}$. Moreover, since the action of $\pi_1(\RP^2)$ on $[\RP^2;\RP^2]_{id}^{\ast}$ exchanges the homotopy classes $[id]^{\ast}$ and $[h_{-1}]^{\ast}$, it follows that the action of $\pi_1(\RP^2)$ on $[K;\RP^2]_{\alpha}^{\ast}$ exchanges $[f]^{\ast}$ and $[f_{-1}]^{\ast}$. Furthermore, we have
$$f_{-1}^{\ast}(\nu)=f^{\ast}\circ h_{-1}^{\ast}(\nu)=-f^{\ast}(\nu).$$

It follows that the five following assertions are equivalent: (i) $[f]^{\ast}$ is fixed by the action of $\pi_1(\RP^2)$ on the set $[K;\RP^2]^{\ast}$; (ii) $[f_{-1}]^{\ast}=[f]^{\ast}$; (iii) $f_{-1}^{\ast}(\nu)=f^{\ast}(\nu)$; (iv) $f^{\ast}(\nu)=0$; (v) $f$ is not strongly surjective. In fact: (i)$\Leftrightarrow$(ii) follows from the definition of the action; (ii)$\Leftrightarrow$(iii) follows from Theorem \ref{Theorem-Bijection}; (iii)$\Leftrightarrow$(iv) follows from the equality above and the fact that group $H^2(K;_{\alpha}\!\Z)$ has odd order (Theorem \ref{Theorem-Odd-Order}), and so none of its elements has order two; (iv)$\Leftrightarrow$(v) follows from Theorem \ref{Theorem-Bijection} and Theorem \ref{Theorem-Exists-f0}.
\end{proof}

More important than the statement of Theorem \ref{Theorem-Action} is its proof.

\begin{corollary}[of the proof of Theorem \ref{Theorem-Action}]\label{Corollary-Action-Classes}
Let $K$ be a two-complex with fundamental group $\Pi$ and such that $H^2(K;\Z)$ is finite of odd order. For each $\alpha\in\hom(\Pi;\Z_2)$, the action of $\pi_1(\RP^2)$ on $[K;\RP^2]_{\alpha}^{\ast}$ corresponds, under the bijection $\Phi:[K;\RP^2]_{\alpha}^{\ast}\to H^2(K;_{\alpha}\!\Z)$, to the multiplication by $-1$ in $H^2(K;_{\alpha}\!\Z)$.
\end{corollary}

We conclude the section with the proof of the Main Theorem (Theorem \ref{Theorem-Main}).

\begin{proof}[Proof of the Main Theorem]
Let $K$ be a two-complex with fundamental group $\Pi$ and such that $H^2(K;\Z)$ is finite of odd order. By Theorem \ref{Theorem-Odd-Order}, for each $\alpha\in\hom(\Pi;\Z_2)$, the twisted cohomology group $H^2(K;_{\alpha}\!\Z)$ is finite of odd order.

Item 1 follows from Theorem \ref{Theorem-Bijection}.

Item 2 follows from Corollary \ref{Corollary-Action-Classes} and Item 1.

For item 3, we consider a non-surjective map $f_0:K\to\RP^2$ inducing $\alpha$ on fundamental groups and such that $f_0^{\ast}=0$, as in Theorem \ref{Theorem-Exists-f0}. We claim that each free homotopy class $[f]\neq[f_0]$ is strongly surjective. Given such a class $[f]$, we can suppose, up to homotopy, that the representative map $f:K\to\RP^2$ is based. Consider the map $f_{-1}:K\to\RP^2$ defined as in the proof of Theorem \ref{Theorem-Action}. Then, for a given map $g:K\to\RP^2$, we have $g\in[f]$ if and only if either $g$ is homotopic to $f$ or $g$ is homotopic to $f_{-1}$. By the proof of Theorem \ref{Theorem-Action}, for $g\in[f]$ one has $g^{\ast}(\nu)=\pm f^{\ast}(\nu)\neq0$, which implies that $g$ is strongly surjective.
\end{proof}


\section{On the order of $H^2(K;_{\alpha}\!\Z)$}\label{Section-Appendix}

\hspace{4mm} We have seen that for a two-complex $K$ with fundamental group $\Pi$ and such that $H^2(K;\Z)$ is finite of odd order, for each $\alpha\in\hom(\Pi;\Z_2)$, the twisted cohomology group $H^2(K;_{\alpha}\!\Z)$ is also finite of odd order. Therefore, it is natural to ask which odd integers can be realized as the order of the group $H^2(K;_{\alpha}\!\Z)$, as we vary $K$, and so $\alpha$ in $\hom(\Pi;\Z_2)$, over a certain family of two-complexes. We approach this question for two subfamilies of the family of all two-complexes $K$ with $H^2(K;\Z)=0$ and whose fundamental group has two generators and a single relation (specifically, model two-complexes induced by presentations of the form $\langle x,y\,|\,r \rangle$, sometimes called $(2,1)$-representations). We start with two examples.
  
\begin{example}\label{Example-H2}
{\rm In \cite{Fenille-Daciberg} it was considered the model two-complex $K_k(1)$ of the group presentation $\Pe=\langle x,y\,|\,x^{k+1}yxy \rangle$, for $k\geq1$ odd. One has $H^2(K_k(1);\Z)=0$ and, for the twisted coefficient system $\beta_2$ which just twists the generator $y$, it was computed $H^2(K_k(1);_{\beta_2}\!\Z)\approx\Z/k\Z$. Therefore, if we consider the family of the complexes $K_k$ as $k$ varies over the positive odd integer, then the set of the positive odd integers is realized by this family.}
\end{example}

\begin{example}\label{Example-H3} 
{\rm Besides Example \ref{Example-H2}, it was  mentioned in \cite[Remark 4.3]{Fenille-Daciberg} that for the two-complex $K_k(2)$ obtained by replacing the word $r_1=x^{k+1}yxy$ by $r_2=x^{k+2+l}y^2x^{-l}$, for $l\geq 0$, we have $H^2(K_k(2);_{\beta_2}\!\Z)\approx \Z/(k+2)\Z$ and again $H^2(K_k(2);\Z)=0$. Hence, for a fixed $k\geq1$ odd, the order of $H^2(\,\cdot\,;_{\beta_2}\!\Z)$ has changed despite the fact that the relators $r_1$ and $r_2$ have the property that the sum $\delta_x(r_1)$ of the powers of $x$ in $r_1$ is the same as the sum $\delta_x(r_2)$ of the powers of $x$ in the $r_2$, and similarly  $\delta_y(r_1)=\delta_y(r_2)$, where $\delta_y(r)$ is the sum of the powers of $y$ in $r$.}
\end{example}

Example \ref{Example-H3} suggests the following question: for a fixed integer $k$, which odd integers can be realized as the order of $H^2(K;_{\alpha}\!\Z)$ as $K$ runs over the family of the model two-complexes of presentations of the form $\langle x,y\,|\,r\rangle$, whose relator $r$ has the property $(\delta_x(r),\delta_y(r))=(k+2,2)$? See Proposition  \ref{Proposition-All-Odd-Orders} below.

Next, we prove that it is possible to realize  all positive odd integers as the order of $H^2(K;_{\alpha}\!\Z)$ as $r$ runs over the words in $F(x,y)$ such that $\delta_x(r)$ and $\delta_y(r)$ are any fixed coprime integers greater than $1$, and $\alpha$ is a surjective homomorphism in $\hom(\Pi;\Z_2)$, one for each group $\Pi$. We remark that in order to have $H^2(K;\Z)=0$, the integers $\delta_x(r)$ and $\delta_{y}(r)$ must be coprime (Equation (\ref{Equation-H2}) in Section \ref{Section-Twisted-Cohomology}).

There are just three homomorphisms from $F(x,y)$ into $\Z_2$, other than the trivial one, namely, $\varphi_1,\varphi_2,\varphi_3:F(x,y)\to\Z_2$ given by: $$\varphi_1(x)=-\varphi_2(x)=\varphi_3(x)=-1 \quad{\rm and}\quad -\varphi_1(y)=\varphi_2(y)=\varphi_3(y)=-1.$$

Hence, each homomorphism $\beta:\Pi\to\Z_2$ is the induced homomorphism on the quotient group $\Pi=F(x,y)/N(r)$ by one of $\varphi_1,\varphi_2,\varphi_3$. If it is defined, we denote by $\beta_i$ the homomorphism induced by $\varphi_i$. Of course, $\beta_i$ exists if and only if $\varphi_i(r)=1$.

Therefore, putting $a=\delta_x(r)$ and $b=\delta_y(r)$, we have: $$\hom(\Pi;\Z_2)=\left\{\begin{array}{cl} \{1,\beta_1\} & \textrm{if} \ a \ \textrm{is even and} \ b \ \textrm{is odd} \\ \{1,\beta_2\} & \textrm{if} \ a \ \textrm{is odd and} \ b \ \textrm{is even} \\ \{1,\beta_3\} & \textrm{if} \ a \ \textrm{and} \ b \ \textrm{are odd} \\ \{1,\beta_1,\beta_2,\beta_3\} & \textrm{if} \ a \ \textrm{and} \ b \ \textrm{are even}
 \end{array}\right..$$

The case $a$ and $b$ even will not be considered, since in this case $H^2(K;\Z)\neq0$. Furthermore, by the same reason, we will consider just $a$ and $b$ coprime.

In what follows, given a word $r\in F(x,y)$, we take $K_{r}$ to be the model two-complex of the group presentation $\langle x,y\,|\,r \rangle$.

Given a pair $(a,b)$ of positive integers, we say that $\beta$ is a {\it feasible homomorphism} for $(a,b)$ provided $\beta\in\hom(\Pi;\Z_2)$, according to the previous description of $\hom(\Pi;\Z_2)$.

\begin{proposition}\label{Proposition-All-Odd-Orders}
Let $c\geq 1$ be odd and $a,b\geq2$ coprime. If $\beta$ is a feasible homomorphism for the pair $(a,b)$, then there exists a word $r\in F(x,y)$ such that $\delta_x(r)=a$, $\delta_y(r)=b$ $($and so $H^2(K;\Z)=0)$ and $H^2(K_r;_{\beta}\!\Z)\approx\Z/c\Z$.
\end{proposition}
\begin{proof}
The proof is divided in three cases, according to the parity of $a$ and $b$.

\vspace{2mm}
{\bf Case 1.} The pair $(a,b)=(odd,even)$, we say $a=2p+3$ and $b=2q+2$. In this case, $\beta_2$ is the unique non-trivial feasible homomorphism for $(a,b)$. We proceed by induction.

Consider the word $r_1=x^{p+2}yx^{p+1}y^{2q+1}$, for which we have $\delta_x(r_1)=a$ and $\delta_y(r_1)=b$. By straight computation, we obtain:
\begin{align*}
\frac{\partial r_1}{\partial x}= & \, (1+x+\cdots x^{p+1})+x^{p+2}y(1+x+\cdots x^{p}) \ \ {\rm and \ so} \ \ \Vert\frac{\partial r_1}{\partial x}\Vert_{\beta_2}=1; \\ \noalign{\smallskip} \frac{\partial r_1}{\partial y}= & \, x^{p+2}+x^{p+2}yx^{p+1}(1+y+\cdots y^{2q}) \ \ {\rm and \ so} \ \ \Vert\frac{\partial r_1}{\partial y}\Vert_{\beta_2}=0.
\end{align*}

It follows that $H^2(K_{r_1};_{\beta_2}\!\Z)=\Z/1\Z=0$ and we have the result for $c=1$.

Suppose, as induction hypotheses, that for a given odd integer $c\geq 1$, there exists $r\in F(x,y)$ such that $\delta_x(r)=a$, $\delta_y(r)=b$ and $H^2(K_{r};_{\beta_2}\!\Z)\approx\Z/c\Z$.

Consider the word $s=r y^{-1}x^{-1}y^{-1}xy^2$. Then $\delta_x(s)=a$, $\delta_y(s)=b$ and we have
\begin{align*}
\frac{\partial s}{\partial x} = & \frac{\partial r}{\partial x} + r y^{-1}(-x^{-1}+x^{-1}y^{-1}) \ {\rm and \ so} \ \Vert\frac{\partial s}{\partial x}\Vert_{\beta}=\Vert\frac{\partial r}{\partial x}\Vert_{\beta}+2; \\ \noalign{\smallskip} \frac{\partial s}{\partial y} = & \frac{\partial r}{\partial y} + r\big(\!-y^{-1}\!+\!\big(y^{-1}x^{-1})(-y^{-1}\!+\!y^{-1}x(1+y)\big)\big) \ {\rm and \ so} \ \Vert\frac{\partial s}{\partial y}\Vert_{\beta}= \Vert\frac{\partial r}{\partial y}\Vert_{\beta}.
\end{align*}

By induction, since it happens with the word $r_1$, we can suppose $\Vert \partial r/\partial y\Vert_{\beta}=0$. Then, by the induction hypothesis, $\Vert \partial r/\partial x\Vert_{\beta}=c$. It follows that $H^2(K_s;_{\beta}\!\Z)=\Z/(c+2)\Z$ and the result follows.

\vspace{2mm}
{\bf Case 2.} The pair $(a,b)=(even,odd)$. In this case, $\beta_1$ is the unique non-trivial feasible homomorphism for $(a,b)$. The proof can be obtained from that of Case 1 by permuting the letters $x$ and $y$.

\vspace{2mm}
{\bf Case 3.} The pair $(a,b)=(odd,odd)$ we say, $a=2p+3$ and $b=2q+3$. In this case, $\beta_3$ is the unique non-trivial feasible homomorphism for $(a,b)$.

Consider the word $r_0=x^{p+2}y^2x^{p+1}y^{2q+1}$, for which we have $\delta_x(r_0)=a$ and $\delta_y(r_0)=b$. By straight computation, we obtain 
$$ \Vert\frac{\partial r_0}{\partial x}\Vert_{\beta_3}=1 \quad {\rm and} \quad \Vert\frac{\partial r_0}{\partial y}\Vert_{\beta_3}=-1.
$$

It follows that $H^2(K_{r_0};_{\beta_3}\!\Z)=\Z/1\Z=0$ and we have the result for $c=1$.

Consider the word $r_1=r_0y^{-1}x^{-1}y^{-1}xy^2$. Then $\delta_x(r_1)=a$, $\delta_y(r_1)=b$ and we have:
$$ \Vert\frac{\partial r_1}{\partial x}\Vert_{\beta_3}=\Vert\frac{\partial r_0}{\partial x}\Vert_{\beta_3}-2=-1 \quad {\rm and} \quad \Vert\frac{\partial r_1}{\partial y}\Vert_{\beta_3}= \Vert\frac{\partial r_0}{\partial y}\Vert_{\beta_3} +2=1.$$

Again $H^2(K_{r_1};_{\beta_3}\!\Z)=\Z/1\Z=0$ and we remain with the result for $c=1$.

Taking it one step further, the order increases.

Take $r_2=r_0(y^{-1}x^{-1}y^{-1}xy^2)^2$. Then $\delta_x(r_2)=a$, $\delta_y(r_2)=b$ and we have:
$$ \Vert\frac{\partial r_2}{\partial x}\Vert_{\beta_3}=\Vert\frac{\partial r_1}{\partial x}\Vert_{\beta_3}-2=-3 \quad {\rm and} \quad \Vert\frac{\partial r_2}{\partial y}\Vert_{\beta_3}= \Vert\frac{\partial r_1}{\partial y}\Vert_{\beta_3} +2=3.$$

Inductively, for a given positive integer $n\geq1$, the word  $r_{n}=r_0(y^{-1}x^{-1}y^{-1}xy^2)^n$ is such that $\delta_x(r_n)=a$, $\delta_y(r_n)=b$ and $$ \Vert\frac{\partial r_n}{\partial x}\Vert_{\beta_3}=-2n+1 \quad {\rm and} \quad \Vert\frac{\partial r_n}{\partial y}\Vert_{\beta_3}= 2n-1.$$ 

It follows that $H^2(K_{r_n};_{\beta_3}\!\Z)=\Z/(2n-1)\Z$ and the result follows.
\end{proof}


\section*{Acknowledgments}

\hspace{4mm} The second and third authors are partially sponsored by Projeto Tem\'atico FAPESP, grant 2016/24707-4: {\it Topologia Alg\'ebrica, Geom\'etrica e Diferencial}. The first author thanks the same project for the support during a visit to the University of S\~ao Paulo.




\rule{4.5cm} {1pt} \vspace{2mm}

 {\sc Marcio Colombo Fenille} ({\sl mcfenille@gmail.com})
 
 Universidade Federal de Uberl\^andia -- Faculdade de Matem\'atica.
 
 Av.\,Jo\~ao Naves de \'Avila, 2121, Santa M\^onica, 38400-902, Uberl\^andia MG, Brasil.
 
 \vspace{4mm}
 
 {\sc Daciberg Lima Gon\c calves} ({\sl dlgoncal@ime.usp.br}) 
 
 Universidade de S\~ao Paulo -- Instituto de Matem\'atica e Estat\'istica.
 
  Rua do Mat\~ao, 1010, Cidade Universit\'aria, 05508-090, S\~ao Paulo SP, Brasil.
 
 \vspace{4mm}
 
 {\sc Oziride Manzoli Neto} ({\sl ozimneto@icmc.usp.br})
  
 Universidade de S\~ao Paulo -- Instituto de Ci\^encias Matem\'aticas e de Computa\c c\~ao.
 
 Av.\,Trabalhador S\~ao-Carlense, 400, Centro, 13566-590, S\~ao Carlos SP, Brasil.

\end{document}